\documentclass{IBS}
\usepackage{pdfsync}
\usepackage[active]{srcltx}
\newcommand{\qed}{\par\vskip-\lastskip
    \vskip-\baselineskip\prbox\par
    \addvspace{12pt plus3pt minus3pt}}
\begin{document}

\markboth{SERGEY BAKULIN} {ON IDENTITIES OF INDICATOR BURNSIDE SEMIGROUPS}

\title{ON IDENTITIES OF INDICATOR BURNSIDE SEMIGROUPS}

\author{SERGEY BAKULIN}

\address{Department of Mathematics and Mechanics\\ St. Petersburg State University\\
Saint-Petersburg, Russia\,\\
\email{vansergen@gmail.com} }

\maketitle

\begin{abstract}

A semigroup variety is said to be a \emph{Rees-–Sushkevich variety} if it is contained in a periodic variety generated by 0-simple semigroups. S. I. Kublanovsky has proven that a variety $V$ is a Rees---Sushkevich variety if and only it does not contain any of special finite semigroups. These semigroups are called \emph{indicator Burnside semigroups}. It is shown that indicator Burnside semigroups have polynomially decidable equational theory. Also it is shown that each indicator Burnside semigroups generate a finitely based variety.

\end{abstract}


\section{Introduction}

In the foundational work by Sushkevich \cite{Sushkevich1}, finite simple semigroups have been characterized in terms of special matrices over finite groups. Then Rees \cite{Rees1} generalized Sushkevich's results to periodic simple semigroups. Namely, he proved that periodic completely 0-simple semigroups can be described by the same construction as used by Sushkevich. Recall that a semigroup is called \emph{0-simple} if it does not have ideals except itself and possibly 0. A 0-simple semigroup is called \emph{completely 0-simple} if it has a minimal non-zero idempotent. Following Kublanovsky \cite{Kublanovsky1} any subvariety of a periodic variety generated by 0-simple semigroups a \emph{Rees-–Sushkevich variety}. Rees---Sushkevich varieties have been studied in a number of articles (see, for instance, \cite{Kublanovsky1,Kublanovsky2,Lee3,Lee6,Mashevitzky2,Reilly1} or Section 9 in the recent survey \cite{Shevrin2}). In particular, as established by Hall et al. \cite{Hall1} the variety $\mathbf{RS_{n}}$ generated by all completely 0-simple semigroups over groups of exponent dividing $n$ is finitely based

$$x^2=x^{n+2},\ xyx=(xy)^{n+1}x,\ (xhz)^{n}xyz=xyz(xhz)^{n}$$.

It is natural to consider the following question: are finitely based Rees---Sushkevich varieties recognizable within the class of all semigroup varieties? In other words: for a given set $\Pi$ of semigroup identities, is it possible to recognize whether or not the variety defined by $\Pi$ is a Rees---Sushkevich variety? Clearly, this question is a special case of the general problem of deducing identities, which is undecidable in the class of all semigroups as was shown by Murskii \cite{Murskii1}. It turns out, the problem of recognizing Rees---Sushkevich varieties is decidable (see Proposition \ref{IBSINTRO}).

This result shows that there exists an algorithm that determines whether the variety defined by a given finite system of identities is a Rees---Sushkevich variety. However, such algorithm has exponential complexity. Thus it is important to understand whether there exists a polynomial algorithm that recognizes Rees---Sushkevich varieties. The following main result of this paper gives an affirmative answer to this question.

\begin{theorem}
  There exists a polynomial algorithm that determines whether the variety defined by a given finite system of identities is a Rees---Sushkevich variety. Specifically, each indicator Burnside semigroup have a polynomially decidable equational theory.
\end{theorem}
Note that the identity-checking problem for a finite semigroup is decidable and is co-NP. \cite{Klima1,Seif1}.

Another major issue in the study of identities is a finite basis property. S. Oates and M. B. Powell \cite{Oates1} proved that each finite group generate a finitely based variety. As opened the first P. Perkins \cite{Perkins2} found a finite semigroup which generates infinitely based variety. His example was a matrix semigroup

$$\left\{\left(%
\begin{array}{cc}
  1 & 0 \\
  0 & 0 \\
\end{array}%
\right),\left(%
\begin{array}{cc}
  0 & 1 \\
  0 & 0 \\
\end{array}%
\right),\left(%
\begin{array}{cc}
  0 & 0 \\
  1 & 0 \\
\end{array}%
\right),\left(%
\begin{array}{cc}
  0 & 0 \\
  0 & 1 \\
\end{array}%
\right),\left(%
\begin{array}{cc}
  0 & 0 \\
  0 & 0 \\
\end{array}%
\right),\left(%
\begin{array}{cc}
  1 & 0 \\
  0 & 1 \\
\end{array}%
\right)\right\},$$
over an arbitrary field. This semigroup is called the 6-element Brandt monoid and denoted by $B_{2}^{1}$. Finite basis property for finite semigroup varieties is being actively studied \cite{Shevrin1, Volkov2}. General problem was posed by A. Tarski \cite{Tarski1}. R. McKenzie proved that Tarski's problem is undecidable in the class of all finite groupoids. The same question in the class of all finite semigroups is still open. So it is interesting question: Do each indicator Burnside semigroup generate a finitely based variety? The second main result of this paper gives an affirmative answer to this question.

\begin{theorem}
  Each indicator Burnside semigroup generate a finitely based variety.
\end{theorem}
Note that, there exists a finite semigroup which has a polynomially decidable equational theory while generates an infinitely based variety \cite{Volkov3}.

\section{Background}

We adopt the standard terminology and notation of semigroup theory
(see \cite{Clifford1}) and universal algebra (\cite{Burris1}). For
reader's convenience, we recall a few basic definitions, notation and
results appeared below.

Denote by $\mathcal{X}$ countably infinite set (the \emph{alphabet}) whose
elements are referred to as \emph{letters}. Let $\mathcal{X^{+}}$ be the \emph{free semigroup} over $\mathcal{X}$. Elements of $\mathcal{X^{+}}$ are referred to as \emph{words}.

If $x$ is a letter and $u$ is a word, then $Occ(x,u)$ denotes the
number of occurrences of $x$ in $u$. If $Occ(x,u)>0$, then we say that
the word $u$ contains $x$. The \emph{content} of $u$ is the set $C(u)$ of
letters occurring in $u$. Denote by $l(u)$ the \emph{length} of $u$, that is the number of letters in $u$ counting multiplicity. The \emph{head}
(respectively, \emph{tail}) of $u$ is the first (respectively, the last) letter in $u$ and it is denoted by $h(u)$ (respectively, by $t(u)$). Further, for a word $u=x_{1}x_{2}\cdots x_{n}$ and an integer $s\leq n$ we denote by
$h_{s}(u)$ the letter $x_{s}$ and by $t_{s}(u)$ the letter
$x_{n-s+1}$; in particular $h_{1}(u)=h(u)$ is the head and
$t_{1}(u)=t(u)$ is the tail of the word $u$.

We write $u\approx v$ to stand for a \emph{semigroup identity}. An identity is \emph{non-trivial} if $u\neq v$ as elements of $\mathcal{X^{+}}$. A
non-trivial identity is called a \emph{permutational identity} if it is of the form $x_{1}x_{2}...x_{n}\approx x_{\pi(1)}x_{\pi(2)}...x_{\pi(n)}$, where
$x_{1},...,x_{n}$ are distinct letters in $X$ and $\pi$ is a
non-trivial permutation of $\{1,...,n\}$. The symmetric group on $n$ symbols is denoted by $S_{n}$.

A letter $x$ is \emph{simple} in the word $u$ if it occurs exactly once in $u$. A word $u$ is simple if all of its letters are simple in it. The set of
all simple letters of a word $u$ is denoted by $S(u)$. The \emph{left core} of a word $u$ is the simple word obtained from $u$ by retaining the
first occurrence of each letter, it is denoted by $LC(u)$. The \emph{right
core} of a word $u$ is defined dually, it is denoted by
$RC(u)$. For example, $LC(x^{6}y^{2}zxt^{2}xt^{7}s)=xyzts$ and
$RC(x^{6}y^{2}zxt^{2}xt^{7}s)=yzxts$.

Let $\Sigma$ be a set of identities. The \emph{deducibility} of an identity $u\approx v$ from the identities in $\Sigma$ is denoted by
$\Sigma\vdash u\approx v$. The \emph{variety} defined by $\Sigma$ is denoted by $V(\Sigma)$. The variety generated by a semigroup $S$ is denoted by $\mathbf{S}$ or $V(S)$. If a variety $V$ satisfies an identity $u\approx v$ we write $V\vDash u\approx v$.

Let $u$ be a word. The word $u$ is called an \emph{isoterm} in the variety $V$ if an identity $u\approx v$ holds in $V$ if and only if $u=v$. The word $u$ is called $(p,q)$\emph{-trivial} if $u$ is a isoterm in the variety defined by the identity $x_{1}x_{2}\cdots x_{p}\approx x_{p+1}^{q}$. An identity $u\approx v$ is called $(p,q)$\emph{-trivial} if words $u,v$ are $(p,q)$-trivial.

We denote by $L_{2}$ (respectively, $R_{2}$) the 2-element left-zero (right-zero) semigroups, by $N_{2}$ and $Y_{2}$ the 2-element semigroup with zero multiplication and 2-element semilattice, respectively, and by $C_{n}$ the cyclic group of order $n$. The cyclic semigroup $\langle c \mid c^{r}=c^{r+d}\rangle$ of index $r$ and period $d$ is denoted by $C_{r,d}$. For convenience, let us denote by $N_{k}$ the semigroup $C_{k,1}$.

Let $S$ be a semigroup. The semigroup $S^{1}$ means the semigroup arising from a semigroup $S$ by adjunction of an identity element $1$, unless $S$ already has an identity, in which case $S^{1}=S$.

\begin{lemma} The following statements holds
\label{MSB}
  \begin{arabiclist}
    \item The variety $\mathbf{L_{2}}$ (respectively, $\mathbf{R_{2}}$) is given by the identity $x\approx xy$ (respectively, $x\approx yx$) and satisfies an identity $u\approx v$ if and only if $h(u)=h(v)$ (respectively, $t(u)=t(v)$).
    \item The variety $\mathbf{L_{2}^{1}}$ (respectively, $\mathbf{R_{2}^{1}}$) is given by the identities $x\approx x^{2}, xy\approx xyx$ (respectively, $x\approx x^{2}, xy\approx yxy$) and satisfies an identity $u\approx v$ if and only if $LC(u)=LC(v)$ (respectively, $RC(u)=RC(v)$).
    \item The variety $\mathbf{Y_{2}}$ is given by the identities $x\approx x^{2}, xy\approx yx$ and satisfies an identity $u\approx v$ if and only if $C(u)=C(v)$.
    \item The variety $\mathbf{C_{n}}$ is given by the identities $x\approx xy^{n}, xy\approx yx$ and satisfies an identity $u\approx v$ if and only if $Occ(x,u)\equiv Occ(x,v)$ mod $n$ for any letter $x\in\mathcal{X^{+}}$.
    \item The variety $\mathbf{N_{2}}$ is given by the identity $xy\approx z^{2}$ and satisfies a non-trivial identity $u\approx v$ if and only if $l(u),l(v)>1$.
    \item The variety $\mathbf{N_{2}^{1}}$ is given by the identities $x^{2}\approx x^{3}, xy\approx yx$ and satisfies an identity $u\approx v$ if and only if $C(u)=C(v)$ and $S(u)=S(v)$.
    \item The variety $\mathbf{N_{3}}$ is given by the identities $xyz\approx w^{3}$ and $xy\approx yx$, and satisfies an identity $u\approx v$ if and only if either $l(u),l(v)\geq3$ or $u\approx v$ is equivalent to $xy\approx yx$.
  \end{arabiclist}
\end{lemma}

\begin{proof}
These result are well-known results and easy to prove.
\end{proof}

An element $a\in S$ is \emph{indecomposable} if the equation $a=xy$ has no
solutions in $S$. It is easy to see that if a finite semigroups $S$ has an indecomposable element then $N_{2}\in \mathbf{S}$.

\begin{lemma}\label{IBSINTRO}(Kublanovsky \cite{Kublanovsky1}, Theorem 1) A semigroup variety $V$ is a Rees---Sushkevich variety if and only if it contains none of the following semigroups:

$A=\langle x,y \mid x=x^{2}; y^{2}=0; xy=yx\rangle$

$B=\langle x,y \mid x^{2}=0; y^{2}=0; xyx=yxy\rangle$

$C_{\lambda}=\langle x,y \mid x^{2}=x^{3}; xy=x; x^{2}y=0; y^{2}=0\rangle$

$C_{\rho}=\langle x,y \mid x^{2}=x^{3}; yx=x; yx^{2}=0; y^{2}=0\rangle$

$N_{3}=\langle x \mid x^{3}=0\rangle$

$D=\langle x,y \mid x^{2}=0; y=y^2; yxy=0\rangle$

$K_{n}=\langle x,y \mid x^{2}=0; y^{2}=y^{n+2}; yxy=0; xy^{q}x=0, (q=2,...,n); xyx=xy^{n+1}x\rangle$

$F_{\lambda}=\langle x,y \mid xy=xyx=xy^{2}; yx=yxy=yx^{2};
x^{2}=x^{2}y=x^{3}; y^{2}=y^{2}x=y^{3} \rangle$

$F_{\rho}=\langle x,y \mid xy=yxy=x^{2}y; yx=xyx=y^{2}x;
x^{2}=yx^{2}=x^{3}; y^{2}=xy^{2}=y^{3}\rangle$

$W_{\lambda}=\langle a,x,y \mid a^{2}=x^{2}=y^{2}=xy=yx=0; ax=axax;
ay=ayay; xa=xaxa; ya=yaya; xay=xax; yax=yay\rangle$

$W_{\rho}=\langle a,x,y \mid a^{2}=x^{2}=y^{2}=xy=yx=0; xa=xaxa;
ya=yaya; ax=axax; ay=ayay; xay=yay; yax=xax\rangle$

$L_{2}^{1}=\langle a,x,y \mid x=x^{2}; y=y^{2}; a=a^{2}; xy=x; yx=y;
ax=xa=x; ay=ya=y\rangle$

$R_{2}^{1}=\langle a,x,y \mid x=x^{2}; y=y^{2}; a=a^{2}; xy=y; yx=x;
ax=xa=x; ay=ya=y\rangle$
\end{lemma}

These semigroups are called \emph{indicator Burnside semigroups}.

\begin{lemma}\label{PERKISCOM}(Perkins \cite{Perkins2}, Theorem ?) Each commutative semigroup generate a finitely based variety.
\end{lemma}

\begin{lemma}\label{SMALLS}(Trakhtman \cite{Trahtman3}, Theorem ?)
  Each semigroup with $\leq5$ elements generate a finitely based variety.
\end{lemma}

\section{Proof of main results}

To prove main results, we need to verify that each of indicator Burnside semigroups generates a finitely based and polynomially recognizable variety. For varieties $\mathbf{L_{2}^{1}},\mathbf{R_{2}^{1}}$ and $\mathbf{N_{3}}$, the desirable conclusion immediately follows from Lemma \ref{MSB}. Other indicator Burnside semigroups are considered below. The section is divided into six subsection.

\subsection{Semigroup $A$}
The variety $\mathbf{A}$ is finitely based because it is commutative (see Lemma \ref{PERKISCOM}). The following statement gives an identity basis of $\mathbf{A}$.

\begin{lemma}\label{AIBS} The variety $\mathbf{A}$ coincides with the variety $\mathbf{N_{2}^{1}}$, whence it is given by the identities $x^{2}\approx x^{3}$ and $xy\approx yx$.
\end{lemma}

\begin{proof} In view of Lemma \ref{MSB}, it is suffices to verify that $\mathbf{A}=\mathbf{N_{2}^{1}}$. It is easy to see that $\mathbf{N_{2}^{1}}\subset\mathbf{A}$. Indeed, $N_{2}^{1}$ is the homomorphic image of the semigroup $A$ under the homomorphism that maps $x$ into $1$, $xy$ and $y$ into $a$, and $0$ into $0$. Suppose that there exists an identity $u\approx v$ such that $\mathbf{N_{2}^{1}}\vDash u\approx v$ but $\mathbf{A}\nvDash u\approx v$. Then there exists a map $\phi:\mathcal{X^{+}}\rightarrow A$ such that $\phi(u)\neq\phi(v)$. But $\sigma(\phi(u))=\sigma(\phi(v))$ because $\mathbf{N_{2}^{1}}\vDash u\approx v$. Therefore, without loss of generality we can assume that $\phi(u)=y$, $\phi(v)=xy$. But the element $y$ is indecomposable
in the semigroup $A$. In this case the word $u$ consists of one
letter $u=x$. But the word $u$ is an isoterm in the variety
$\mathbf{N_{2}^{1}}$. Therefore, $u\equiv v$ and $\mathbf{A}$ satisfies the identity $u\approx v$. This contradiction completes the proof.
\end{proof}

\subsection{Semigroup $B$}

The semigroup $B$ is a 4-nilpotent semigroup, and it is easy to verify
that each finite nilpotent semigroup is finitely based and generate a
variety whose finite membership problem admits a linear algorithm.

\begin{lemma}\label{BIBSIB} The variety $\mathbf{B}$ satisfies a non-trivial identity $u\approx v$ if and only if one of the following statements hold:

\begin{arabiclist}
  \item Words $u$, $v$ are not $(4,2)$-trivial,
  \item $u\approx v$ has the form $abc\approx cba$,
  \item $u\approx v$ has the form $aba\approx bab$.
\end{arabiclist}
\end{lemma}

\begin{proof}\textbf{Necessity.} Assume that identity $u\approx v$ holds in the variety $\mathbf{B}$. It is easy to see that for all $x_{1}, x_{2}, x_{3}, x_{4}, x_{5}\in B$ we have $x_{1}x_{2}x_{3}x_{4}=x_{5}^{2}=0$. Therefore, if there exists a homomorphism such that $\phi(u)\neq0$ then $l(u)\leq4$ and $x^{2}$ is not a subword of $u$ for all $x\in\mathcal{X}$. Hence, $u$ is $(4,2)$-trivial if and only if $v$ is $(4,2)$-trivial too. Assume that the statement 1 does not hold. Then it is easy to see that words $u$ and $v$ has the same length, because words $u$, $v$ are not $(p,q)$-trivial. We observe also that the word $ab$ is isoterm in the variety $\mathbf{B}$. Thus, we can suppose that $l(u)=l(v)=3$. In this case we find that $C(u)=C(v)$. First, suppose that the word $u$ is simple. Then $u$, $v$ are products of three different letter, $u=x_{1}x_{2}x_{3}$ and $v=x_{\pi(1)}x_{\pi(2)}x_{\pi(3)}$. If $\pi(2)\neq2$ then $x_{1}x_{3}$ or $x_{3}x_{1}$ is a subword of $v$. In this case we can consider the homomorphism, such that:
$$\phi(a)=\begin{cases}
x, &\text{if either }  a=x_{1} \text{ or } a=x_{3};\\
y, &\text{otherwise}
\end{cases}$$
Then we can note that $\phi(v)=xxy$ or $\phi(v)=yxx$. In both cases we find that $\phi(v)=0$ but $\phi(u)=xyx$. A contradiction. This shows that $v=x_{3}x_{2}x_{1}$. This means that the identity $u\approx v$ has the form $abc\approx cba$.

It remains to consider the case when the word $u$ is not simple. Then $u$ equals $x_{1}x_{2}x_{1}$. Since $C(u)=C(v)$ and the identity $u\approx v$ is not trivial, we can observe that the word $v$ equals $x_{2}x_{1}x_{2}$. This means that the identity $u\approx v$ has the form $aba\approx bab$.

\textbf{Sufficiency.} It is easy to see that if statements 2 or 3 hold then the identity $u\approx v$ holds in the variety $\mathbf{B}$. If the statement 1 holds then $\phi(u)=\phi(v)$ for all homomorphisms $\phi:\mathcal{X}\rightarrow B$. In other words, the identity $u\approx v$ holds in the variety $\mathbf{B}$.
\end{proof}

Lemma \ref{BIBSIB} allows to find an identity basis of the variety $\mathbf{B}$.

\begin{lemma}\label{BIBSB} The identities
  \begin{align}
\label{a2=bcde}
& a^{2}\approx bcde,\\
\label{abc=cba}
& abc\approx cba,\\
\label{aba=bab}
& aba\approx bab
\end{align}
form an identity basis of the variety $\mathbf{B}$.
\end{lemma}

\begin{proof} Suppose that the identity $u\approx v$ holds in the variety $\mathbf{B}$. We can apply Lemma \ref{BIBSIB}. If the identity $u\approx v$ satisfies statement 1 then it follows from the identity $a^{2}\approx bcde$. If the identity $u\approx v$ satisfies statement 2 or 3 then it follows from the identity \eqref{abc=cba} or \eqref{aba=bab} respectively. Hence identities \eqref{a2=bcde},\eqref{abc=cba} and \eqref{aba=bab} constitute an identity basis of the variety $\mathbf{B}$.
\end{proof}

\subsection{Semigroups $C_{\lambda}$ and $C_{\rho}$}

The variety $\mathbf{C_{\lambda}}$ is finitely based because $C_{\lambda}$ consists of 5 elements, while each semigroup with $\leq5$ elements generates a finitley based variety by Lemma \ref{SMALLS}. Next we find a finite identity basis of the variety $\mathbf{C_{\lambda}}$ and show that it is polynomially recognizable.

By duality, we consider only the semigroup $C_{\lambda}$.

\begin{lemma}
\label{CIBSID}
Let $u$, $v$ be words. The following are equivalent:
\begin{arabiclist}
  \item The variety $\mathbf{C_{\lambda}}$ satisfies the identity $u\approx v$;
  \item The following are satisfied:
  \begin{romanlist}
        \item $C(u)=C(v)$,
        \item $Occ(h(u),u)=1$ if and only if $Occ(h(v),v)=1$,
        \item if $Occ(h(u),u)=1$ then $h(u)=h(v)$,
        \item $Occ(h_{2}(u),u)=1$ if and only if $Occ(h_{2}(v),v)=1$,
        \item if $Occ(h_{2}(u),u)=1$ then $h(u)=h(v)$ and $h_{2}(u)=h_{2}(v)$.
  \end{romanlist}
  \item The following are satisfied:
  \begin{romanlist}
      \item $C(u)=C(v)$.
  \item one of the following statement hold:
  \begin{alphlist}
  \item $h(u)=h(v)$, $h_{2}(u)=h_{2}(v)$ and $Occ(h(u),u)=Occ(h(v),v)=Occ(h_{2}(u),u)=Occ(h_{2}(v),u)=1$,
  \item $h(u)=h(v)$, $h_{2}(u)=h_{2}(v)$, $Occ(h_{2}(u),u)=Occ(h_{2}(v),v)=1$ and $Occ(h(u),u)$ $,Occ(h(v),v)>1$,
  \item $h(u)=h(v)$, $Occ(h(u),u)=Occ(h(v),v)=1$ and $Occ(h_{2}(u),u)=Occ(h_{2}(v),v)>1$,
  \item $Occ(h(u),u),Occ(h(v),v),Occ(h_{2}(u),u),Occ(h_{2}(v),v)>1$.
  \end{alphlist}
  \end{romanlist}
\end{arabiclist}
\end{lemma}
\begin{proof} The equivalence of the statements 2 and 3 is verified easily.

$1\Longrightarrow2$. Assume that an identity $u\approx v$ holds in the variety $\mathbf{C_{\lambda}}$. It easy to see that $Y_{2}\in\mathbf{C_{\lambda}}$. Therefore the statement (i) holds by Lemma \ref{MSB}. Suppose that $Occ(h(u),u)=1$ but $Occ(h(v),v)>1$. Suppose that $h(u)\neq h(v)$. Consider a homomorphism such that
$$\phi_{1}(a)=\begin{cases}
xy, &\text{if }  a=h(u);\\
x^{2}, &\text{otherwise}
\end{cases}$$

Then $\phi_{1}(u)=xyx^{2k}=xy$ and $\phi_{1}(v)=0$. We obtain that $\phi_{1}(u)=xy\neq0=\phi_{1}(v)$ in contradiction with assumption that $\mathbf{C_{\lambda}}\vDash u\approx v$.

Let now $h(u)=h(v)$. Then $\phi_{1}(u)=xyx^{2k}=xy$ and $\phi_{1}(v)=0$ . But $\phi_{1}(u)=xy\neq0=\phi_{1}(v)$ in contradiction with assumption that $\mathbf{C_{\lambda}}\vDash u\approx v$.

We show that $Occ(h(u),u)=1$ if and only if $Occ(h(v),v)=1$. Let us observe that if $Occ(h(u),u)=1$ and $h(u)\neq h(v)$ then $\phi_{1}(u)=xy\neq0=\phi(v)$. Hence (iii) holds.

Proof of the conditions \emph{(ii)}, \emph{(iii)} is similar to \emph{(iv)}, \emph{(v)}. Unless we should consider another map. To prove \emph{(iv)}, \emph{(v)} it is sufficiently to consider the following homomorphism
$$ \phi_{2}(a)=\begin{cases}
y, &\text{if }  a=h_{2}(u);\\
x, &\text{otherwise}
\end{cases}
$$
$2\Longrightarrow1$. Let an identity $u\approx v$ satisfies conditions \emph{(i)}-\emph{(v)}. Let $\phi$ be the map from $\mathcal{X^{+}}$ to $C_{\lambda}$. Since $Y_{2}\in \mathbf{C_{\lambda}}$ so we can assume that $\phi(a)\neq0$ for all $a\in\mathcal{X}$. The identity $u\approx v$ is not trivial. Thus $l(u),l(v)\geq3$ follows from conditions \emph{(ii)}, \emph{(iii)}, \emph{(iv)}, \emph{(v)}. It is easy to see that $C_{\lambda}xy=\{0\}$ and $C_{\lambda}C_{\lambda}y=\{0\}$. Therefore if $\phi(a)=xy$ (respectively, $\phi(a)=y$) for the letter $a\in\mathcal{X}$ such that $a\neq h(u)$ (respectively, $a\neq h(u),h_{2}(u)$) then we have $\phi(u)=\phi(v)=0$.

Suppose that $\phi(h(u))\in\{y, xy\}$. If there exists a letter $b\in C(u)\backslash\{h(u)\}$ such that $\phi(b)\notin \{x, x^{2}\}$ then $\phi(u)=\phi(v)=0$. Otherwise, conditions \emph{(ii)}, \emph{(iii)} imply that $\phi(u)=\phi(v)=\phi(h(u))$.

Let now $\phi(h(u))\in\{x, x^{2}\}$. If $\phi(h_{2}(u))=y$. Then $Occ(h_{2},u)=Occ(h_{2}(v),v)=1$ and $h_{2}(u)=h_{2}(v), h(u)=h(v)$. If $\phi(h(u))=x^{2}$ or there exist a letter $b\in C(u)\backslash\{h(u),h_{2}(u)\}$ such that $\phi(b)\notin\{x, x^{2}\}$ then $\phi(u)=\phi(v)=0$. Otherwise conditions \emph{(iii)},\emph{(iv)} imply that $\phi(u)=\phi(v)=xy$.

It remains to consider the case when $\phi(h_{2}(u))\neq y$. Then $\phi(h_{2}(u))\in\{x,x^{2}\}$. Consider a letter $b\in C(u)\backslash\{h(u),h_{2}(u)\}$. Note that we discussed cases $\phi(b)\in\{xy,y,0\}$. So $\phi(b)\in\{x,x^{2}\}$ and we have $\phi(u)=\phi(v)=x^{2}$.

We obtain $\phi(u)=\phi(v)$ in all cases. Hence the identity $u\approx v$ holds in the variety $\mathbf{C_{\lambda}}$.
\end{proof}

Now, we are ready to find an identity basis of the variety $\mathbf{C_{\lambda}}$.

\begin{lemma}\label{CIBSB} The identities
\begin{align}
\label{a2=a3}
& a^{2}\approx a^{3}\\
\label{a2b=b2a}
& a^{2}b\approx b^{2}a\\
\label{abc=abc2}
& abc\approx abc^{2}
\end{align}
form a identity basis of the variety $\mathbf{C_{\lambda}}$.
\end{lemma}

\begin{proof}
Note that the identity
\begin{equation}
\label{abcd=abdc}
abcd\approx abdc
\end{equation}
follows from identities \eqref{a2b=b2a}, \eqref{abc=abc2}. Indeed,
$$abcd\approx^{\eqref{abc=abc2}}abc^{2}d\approx^{\eqref{a2b=b2a}}abd^{2}c\approx^{\eqref{abc=abc2}}abdc$$

Also the identity
\begin{equation}
\label{a2b2=b2a2}
a^{2}b^{2}\approx b^{2}a^{2}
\end{equation}
follows from identities \eqref{a2b=b2a}, \eqref{abc=abc2}:

$$a^{2}b^{2}\approx^{\eqref{abc=abc2}}a^{2}b\approx^{\eqref{a2b=b2a}}b^{2}a\approx^{\eqref{abc=abc2}}b^{2}a^{2}$$

Consider an identity $u\approx v$ such that $\mathbf{C_{\lambda}}\vDash u\approx v$. We can apply Lemma \ref{CIBSID}. Consider all possible casses:

\begin{arabiclist}
  \item $h(u)=h(v)$, $h_{2}(u)=h_{2}(v)$. Then the identity $u\approx v$ has the form $abu'\approx abv'$ such that $\mathbf{Y_{2}}\vDash u'\approx v'$. In this case the identity $u'\approx v'$ follows from identities $a\approx a^{2},ab\approx ba$. Hence, the identity $abu'\approx abv'$ follows from identities \eqref{abc=abc2},\eqref{abcd=abdc}.
    \item $h(u)=h(v)$, $h_{2}(u)\neq h_{2}(v)$. Then the identity $u\approx v$ has the form $abu'bu''\approx acv'cv''$. Applying identities \eqref{abc=abc2}, \eqref{abcd=abdc} we find that the identity $u\approx v$ is equivalent to $ab^{2}c^{2}u'''\approx ac^{2}b^{2}v'''$ where $\mathbf{Y_{2}}\vDash u'''\approx v'''$. Hence, the identity $ab^{2}c^{2}u'''\approx ac^{2}b^{2}v'''$ follows from identities \eqref{abc=abc2}, \eqref{abcd=abdc}, \eqref{a2b2=b2a2}.
    \item $h(u)\neq h(v)$. Then applying identities \eqref{a2=a3}, \eqref{a2b=b2a}, \eqref{abc=abc2}, \eqref{a2b2=b2a2} we can find that the identity $u\approx v$ is equivalent to $x_{1}^{2}...x_{k}^{2}\approx x_{\pi(1)}^{2}...x_{\pi(k)}^{2}$ where $\pi\in S_{k}$. But this identity follows from \eqref{a2b2=b2a2}. Therefore, the identity $u\approx v$ follows from identities \eqref{a2=a3}, \eqref{a2b=b2a}, \eqref{abc=abc2}.
\end{arabiclist}

We have proved that an identity $u\approx v$ follows from identities \eqref{a2=a3},\eqref{a2b=b2a} and \eqref{abc=abc2} whenever it holds in the variety $\mathbf{C_{\lambda}}$. This evidently implies the desirable conclusion.
\end{proof}

\subsection{Semigroups $K_{n}$ and $D$}

Let $V$ a variety defined by a given finite system of identities. To check that the variety $V$ is a Rees-Suchkevich variety we have to verify that $K_{n}\notin V$. But we have to do it just for such $n$ is dividing a period of $V$. Hence we verify it for $n=1$. It is not difficult to see that $\mathbf{D}\subseteq \mathbf{K_{1}}$ because the semigroup $D$ is a quotient of the semigroup $K_{1}$. But it will be easy to see that in fact $\mathbf{D}=\mathbf{K_{1}}$. Whence we can assume that the semigroup $D$ is unnecessary in the list of indicator Rees---Sushkevich semigroups.

Let the word $u$ have the following form $u=aWb$. Denote by $\xi_{n}$ the map from $\mathcal{X^{+}}$ to $\mathbb{Z}$ defined by the rule

$$ \xi_{n}(u)=\begin{cases}
0, &\text{if }  l(u)=2;\\
0, &\text{if either }  a\in C(W) \text{ or } b\in C(W);\\
\gcd\limits_{x_{i}\in C(W)}\{Occ(x_{i},W)\}, &\text{otherwise}.
\end{cases}
$$

\begin{lemma}
\label{KIBSID}
The variety $\mathbf{K_{n}}$ satisfies a non-trivial identity $u\approx v$ if and only if the following conditions are satisfied:

\begin{arabiclist}
  \item $C(u)=C(v)$
  \item $\mathbf{C_{2,n}}\vDash u\approx v$
  \item one of the following conditions are satisfied:
  \begin{alphlist}
    \item $h(u)=h(v), Occ(h(u),u)=Occ(h(v),v)=1, Occ(t(u),u),Occ(t(v),v)>1$
    \item $t(u)=t(v), Occ(t(u),u)=Occ(t(v),v)=1, Occ(h(u),u),Occ(h(v),v)>1$
    \item $h(u)=h(v),t(u)=t(v), Occ(h(u),u)=Occ(h(v),v)=Occ(t(u),u)=Occ(t(v),v)=1$
    \item $h(u)=h(v)=t(u)=t(v), Occ(h(u),u)=Occ(h(v),v)=2, \xi_{n}(u)=\xi_{n}(v)=1$
    \item $Occ(h(u),u),Occ(h(v),v),Occ(t(u),u),Occ(t(v),v)>1, \xi_{n}(u),\xi_{n}(v)\neq1$
  \end{alphlist}
\end{arabiclist}
\end{lemma}

\begin{proof}\textbf{Necessity.} Assume that an identity $u\approx v$ holds in the variety $\mathbf{K_{n}}$. It is easy to see that $Y_{2}\in \mathbf{K_{n}}$. Therefore $C(u)=C(v)$ by Lemma \ref{MSB}. The semigroup generated by the element $y$ is isomorphic to the semigroup $C_{2,n}$. So the identity $u\approx v$ holds in the variety $\mathbf{C_{2,n}}$. To prove \emph{3} consider possible cases:

\begin{arabiclist}
  \item $Occ(h(u),u)=1,Occ(t(u),u)>1$. Consider the following map
    $$\phi(a)=\begin{cases}
    x, &\text{if }  a=h(u);\\
    y^{2n}, &\text{otherwise}
    \end{cases}
    $$

We obtain $\phi(u)=xy^{2n}$. But the identity $u\approx v$ holds in the variety $\mathbf{K_{n}}$, and thus $\phi(v)=xy^{2n}$. It follows that $Occ(h(v),v)=1$ and $h(u)=h(v)$.

\item $Occ(t(u),u)=1,Occ(h(u),u)>1$. Consider the following map
$$\phi(a)=\begin{cases}
x, &\text{if }  a=t(u);\\
y^{2n}, &\text{otherwise}
\end{cases}
$$
As in the previous case we find that $Occ(t(v),v)=1$ and $t(u)=t(v)$.

Comparing cases \emph{1} and \emph{2} we see that $Occ(t(v),v)>1$ and Condition \emph{(a)} holds in the first case, while $Occ(h(v),v)>1$ and Condition \emph{(b)} holds in the second case.

\item $Occ(h(u),u)=Occ(t(u),u)=1$. Note that $N_{2}\in \mathbf{C_{2,n}}\subseteq \mathbf{K_{n}}$. So $l(u),l(v)>1$ because the identity $u\approx v$ is not trivial. Applying the same arguments as in previous cases we find that $Occ(h(v),v)=Occ(t(v),v)=1$ and $h(u)=h(v),t(u)=t(v)$. Hence Statement \emph{(c)} holds.

\item $h(u)=t(u), Occ(h(u),u)=2$. If $l(u)=2$ then we have $\xi_{n}(u)=0$. It follows from the previous cases that the word $ab$ is isoterm. Therefore the word $v$ has the form $v=a^{2+kn}$. Hence Condition \emph{(e)} holds. Now assume that the word $u$ has the form $u=aWa$, where the word $W$ is non-simple. Consider possible cases:

    \begin{romanlist}
      \item $\xi_{n}(u)=1$. Then there exists a map $\phi':\mathcal{X^{+}}\rightarrow C_{2,n}\leq K_{n}$ such that $\phi'(W)=y^{2n+1}$. Consider the following map:
    $$\phi(a)=\begin{cases}
    a, &\text{if }  a\in C(W);\\
    x, &\text{otherwise}
    \end{cases}
    $$
    Then we have $\phi(u)=xy^{n+1}x=xyx$ and $\phi(v)=xyx$. But in this case we can conclude that the word $v$ has the form $v=aW'a$, where $\xi_{n}(v)=1$. Thus Condition \emph{(d)} holds.
    \item $\xi_{n}(u)\neq1$. Then for all maps we have $\phi(u)\neq xyx$. Thus we have $\xi_{n}(v)\neq1$ and $Occ(h(v),v),Occ(t(v),v)>1$. Therefore Condition \emph{(e)} holds.
    \end{romanlist}
    \item $Occ(h(u),u),Occ(t(u),u)>1$ and $\xi_{n}(u)\neq1$. Now, we can apply all previous arguments to conclude that $Occ(h(v),v),Occ(t(v),v)>1$ and $\xi_{v}(u)\neq1$. Thus Condition \emph{(e)} holds.
\end{arabiclist}

\textbf{Sufficiency.} Consider an identity $u\approx v$ that satisfies conditions \emph{1},\emph{2} and \emph{3}. Note that the semigroup $K_{n}=\cup_{i=1}^{6}F_{i}$, where

$F_{1}=\{x\}$

$F_{2}=\{y^{k} \mid 1\leq k \leq n+1\}$

$F_{3}=\{xy^{k} \mid 1\leq k \leq n+1\}$

$F_{4}=\{y^{k}x \mid 1\leq k \leq n+1\}$

$F_{5}=\{xyx\}$

$F_{6}=\{0\}$
and $F_{i}\cap F_{j}=\emptyset$ if and only if $i\neq j$.

Consider a map $\phi: \mathcal{X^{+}}\rightarrow K_{n}$. Condition \emph{(3)} implies  that $\phi(u),\phi(v)\in F_{i}$ for some $i$.  In turn, Conditions \emph{(1)},\emph{(2)} imply that values of words $u$ and $v$ coincide in $F_{i}$. Thus, the identity $u\approx v$ holds in the variety $\mathbf{K_{n}}$.\end{proof}

Now, we can prove that the variety $\mathbf{K_{n}}$ is finitely based.

\begin{lemma}\label{KIBSFB} The variety $\mathbf{K_{n}}$ is given by the identity \eqref{a2b2=b2a2} and the following identities:

\begin{align}
\label{a2=an}
& a^{2}\approx a^{n+2},\\
\label{abcd=acbd}
& abcd\approx acbd,\\
\label{abc=abn+1c}
& abc\approx ab^{n+1}c,\\
\label{abma=an+1bma}
& ab^{m}a\approx a^{n+1}b^{m}a, \text{ where } \gcd(m,n)>1 \text{ and } m\leq n.
\end{align}
\end{lemma}

\begin{proof} Consider an identity $u\approx v$ such that $\mathbf{K_{n}}\vDash u\approx v$. We can apply Lemma~\ref{KIBSID}. Note that identities $a^{k+1}b^{l+1}\approx b^{l+1}a^{k+1}$ follow from identities \eqref{a2b2=b2a2}, \eqref{abcd=acbd} for all $k,l\geq1$. Consider possible cases:
  \begin{arabiclist}
    \item $h(u)=h(v), Occ(h(u),u)=Occ(h(v),v)=1, Occ(t(u),u),Occ(t(v),v)>1$. Then the identity $u\approx v$ has the form $au'bu''b\approx av'cv''c$. Applying identities \eqref{a2b2=b2a2}, \eqref{abcd=acbd} we can get the identity $a\overline{u}b^{k}\approx a\overline{v}b^{l}$ where $b\notin C(\overline{u}),C(\overline{v})$. Now applying the identity \eqref{a2=an} we get the identity $a\overline{u}b^{k}\approx a\overline{v}b^{k}$. But the identity $\overline{u}\approx \overline{v}$ holds in the variety $\mathbf{C_{n}}\vee\mathbf{Y_{2}}$. Hence, it follows from identities $b\approx b^{n+1}, bc\approx cb$. Therefore the identity $u\approx v$ follows from identities \eqref{a2b2=b2a2}, \eqref{abcd=acbd}, \eqref{abc=abn+1c}.
    \item $t(u)=t(v), Occ(t(u),u)=Occ(t(v),v)=1, Occ(h(u),u),Occ(h(v),v)>1$. This case can be considered in the same way as the previous case by duality.
    \item $h(u)=h(v),t(u)=t(v), Occ(h(u),u)=Occ(h(v),v)=Occ(t(u),u)=Occ(t(v),v)=1$. Then the identity $u\approx v$ has the form $au'b\approx av'b$. But the identity $u'\approx v'$ holds in the variety $\mathbf{C_{n}}\vee\mathbf{Y_{2}}$. Thus, it follows from identities $b\approx b^{n+1}, bc\approx cb$. Therefore the identity $u\approx v$ follows from identities \eqref{abcd=acbd}, \eqref{abc=abn+1c}.
    \item $h(u)=h(v)=t(u)=t(v), Occ(h(u),u)=Occ(h(v),v)=2, \xi_{n}(u)=\xi_{n}(v)=1$. Then the identity $u\approx v$ has the form $au'a\approx av'a$. But in this case the identity $u'\approx v$ holds in the variety $\mathbf{C_{n}}\vee\mathbf{Y_{2}}$. Thus, we can apply the same argument as in the previous case. Therefore the identity $u\approx v$ follows from identities \eqref{abcd=acbd}, \eqref{abc=abn+1c}.
    \item $Occ(h(u),u),Occ(h(v),v),Occ(t(u),u),Occ(t(v),v)>1,\xi_{n}(u),\xi_{n}(v)\neq1$. Note that applying identities \eqref{abcd=acbd}, \eqref{abma=an+1bma} to the identity $u\approx v$ we can get an identity $u'\approx v'$ such that for any letter $x\in C(u')=C(v')$ we have $Occ(x,u'),Occ(x,v')>1$. Now applying identities $a^{k+1}b^{l+1}\approx b^{l+1}a^{k+1}$ we get the identity $x_{1}^{k_{1}}x_{2}^{k_{2}}...x_{q}^{k_{p}}\approx x_{1}^{l_{1}}x_{2}^{l_{2}}...x_{q}^{l_{p}}$ which follows from the identity \eqref{a2=an}.
  \end{arabiclist}
Hence identities \eqref{a2b2=b2a2},\eqref{a2=an},\eqref{abcd=acbd},\eqref{abc=abn+1c}\eqref{abma=an+1bma} constitute an identity basis of the variety $\mathbf{K_{n}}$.
\end{proof}

\subsection{Semigroups $F_{\lambda}$ and $F_{\rho}$}

By duality, it suffices to consider only the semigroup $F_{\lambda}$.

\begin{lemma}
\label{FIBSID}
The variety $\mathbf{F_{\lambda}}$ satisfies a non-trivial identity $u\approx v$ if and only if $h(u)=h(v)$ and $h_{2}(u)=h_{2}(v)$
\end{lemma}

\begin{proof}\textbf{Necessity.} Assume that the identity $u\approx v$ holds in the variety $\mathbf{F_{\lambda}}$. Note that the subset $\{xy, yx\}$ constitutes a semigroup isomorphic to $L_{2}$. Thus the equality $h(u)=h(v)$ holds. Now we are going to show that $h_{2}(u)=h_{2}(v)$. Indeed, otherwise we may assume without loss of generality that $h_{2}(u)\neq h(u)$ and consider the following map
$$\phi(a)=\begin{cases}
x, &\text{if }  a=h(u);\\
y, &\text{if }  a=h_{2}(u);\\
x^{2}, &\text{otherwise}
\end{cases}$$
Then we have $\phi(u)=xy,\phi(v)=x^{2}$. But it is in contradiction with hypothesis that the identity $u\approx v$ holds in the variety $\mathbf{F_{\lambda}}$. Hence $h_{2}(u)=h_{2}(v)$.

\textbf{Sufficiency.} Consider an identity $u\approx v$ such that $h(u)=h(v)$ and $h_{2}(u)=h_{2}(v)$. Consider any map $\phi:\mathcal{X^{+}}\rightarrow F_{\lambda}$. Then $\phi(h(u)h_{2}(u))=\phi(h(v)h_{2}(v))\in\{x^{2},y^{2},xy,yx\}$. But $\phi(h(u)h_{2}(u)),\phi(h(v)h_{2}(v))$ is a left zero in the semigroup $F_{\lambda}$. Hence $\phi(u)=\phi(h(u)h_{2}(u))=\phi(h(v)h_{2}(v))=\phi(v)$. Therefore the identity $u\approx v$ holds in the variety $\mathbf{F_{\lambda}}$.
\end{proof}

Lemma \ref{FIBSID} readily implies the following

\begin{lemma}\label{FIBSFB} The identity
\begin{equation}\label{ab=abc}
ab\approx abc
\end{equation}
forms an identity basis of the variety $\mathbf{F_{\lambda}}$.\qed
\end{lemma}

\subsection{Semigroups $W_{\lambda}$ and $W_{\rho}$}

By duality, we consider only the semigroup $W_{\lambda}$.

To prove results in this section we need the following construction: each word $u$ is associated with an undirected graph $Gr(u)$ with
vertex set $C(u)\times \{0,1\}$ connected as follows: we draw an edge from $(x,0)$ to $(y,1)$ if and only if $xy$ is a factor of $u$.

The semigroup $B_{2}$ plays crucial role here. A identity basis of the variety $\mathbf{B_{2}}$ was found by Trahtman \cite{Trahtman1}. But his proof has a gap. Reilly \cite{Reilly1} reproved Trahtman's resuls. A solution to the word problem for $B_{2}$ was first provided by Mashevitsky \cite{Mashevitzky5}. Reilly \cite{Reilly1} gave another proof of  Mashevitsky's result in the most convenient terms. His solution can be stated as follows:

\begin{lemma}\cite[Theorem~5.1]{Reilly1}
\label{B2ID}
The variety $\mathbf{B_{2}}$ satisfies an identity $u\approx v$ if and only if the following conditions are satisfied:

\begin{arabiclist}
  \item $C(u)=C(v)$,
  \item the graphs $Gr(u),Gr(v)$ have the same connected components,
  \item the vertices $(h(u),1),(h(v),1)$ lie in the the same connected component,
  \item the vertices $(t(u),0),(t(v),0)$ lie in the the same connected component.
\end{arabiclist}
\end{lemma}

We need also the following corollary:

\begin{corollary}\label{B2IDC} Let the identity $u\approx v$ hold in the variety $\mathbf{B_{2}}$ and the word $u$ has the form $u=u'xu''$ where $x\notin C(u'u'')$ and $C(u')\cap C(u'')=\emptyset$. Then the word $v$ has the form $v=v'xv''$ and identities $u'\approx v', u''\approx v''$ hold in the variety $\mathbf{B_{2}}$.
\end{corollary}

\begin{proposition}\cite[Corollary~9.2]{Reilly1}\label{B2L2FB} The variety $\mathbf{L_{2}}\vee\mathbf{B_{2}}$ is defined by \eqref{a2=a3} and identities
\begin{align}
\label{aba=ababa}
& aba\approx ababa,\\
\label{ab2c2=ac2b2}
& ab^{2}c^{2}\approx ac^{2}b^{2}.
\end{align}
\end{proposition}

The semigroup $B_{2}$ can be presented as

$$B_{2}=\{(i,j) \mid i,j\in\{0,1\}\}\cup\{0\}$$

with the following binary operation

$$(i,j)(k,l)=\begin{cases}
(i,l), &\text{if }  j=k,\\
0, &\text{otherwise.}\end{cases}$$.

Denote by $\Gamma$ the set

$$\{(i,j,k,1) \mid i,j,k\in\{0,1\}\}\cup\{(0,1,2,0),(1,0,0,0),(1,0,1,0)\}\cup\{0\}$$

and define a binary operation by the following formula:

$$(i_{1},j_{1},k_{1},l_{1})(i_{2},j_{2},k_{2},l_{2}),=\begin{cases}
(i_{1},j_{2},k_{1},1), &\text{if }  j_{1}=i_{2}, k_{1}\neq2,\\
(i_{1},j_{2},k_{2},1), &\text{if }  j_{1}=i_{2}, k_{1}=2.\\
0, &\text{otherwise}\end{cases}$$

It is easy to check that this operation is associative i.e. $\Gamma$ is a semigroup. More precisely $\Gamma$ is isomorphic to $W_{\lambda}$. Indeed, isomorphism can be define on the generators as follows: $a\mapsto (0,1,2,0)$, $x\mapsto (1,0,0,0)$, $y\mapsto (1,0,1,0)$, $0\mapsto 0$.

It is easy to see that the semigroup $B_{2}$ is a homomorphic image of $\Gamma$. Define the homomorphism $\tau:\Gamma\rightarrow B_{2}$ on the generators as follows: $(0,1,2,0)\mapsto (0,1)$, $(1,0,0,0)\mapsto (1,0)$, $(1,0,1,0)\mapsto (1,0)$, $0\mapsto 0$.

\begin{lemma}
\label{WIBSID}
The variety $\mathbf{W_{\lambda}}$ satisfies an identity $u\approx v$ if and only if the following conditions are satisfied:

\begin{arabiclist}
  \item $\mathbf{L_{2}}\vee\mathbf{B_{2}}\vDash u\approx v,$
  \item if $(h(u),0),(h(u),1)$ does not lie in the same connected component of $Gr(u)$ then $h_{2}(u)=h_{2}(v).$
\end{arabiclist}
\end{lemma}

\begin{proof}\textbf{Necessity.} Let the identity $u\approx v$ holds in the variety $\mathbf{W_{\lambda}}$. Elements $(0,0,0,1),(0,0,1,1)$ constitute subsemigroup of $W_{\lambda}$ isomorphic to $L_{2}$. The semigroup $B_{2}$ is a homomorphic image of $W_{\lambda}$. Hence the identity $u\approx v$ holds in the variety $\mathbf{L_{2}}\vee\mathbf{B_{2}}$. Let us show that if $(h(u),0),(h(u),1)$ does not lay in the same connected component of $Gr(u)$ then $h_{2}(u)=h_{2}(v)$. Assume the converse. I.e. $(h(u),0),(h(u),1)$ does not lie in the same connected component of $Gr(u)$ and $h_{2}(u)\neq h_{2}(v)$. Then from Lemma \ref{B2ID} it is follows that $h(u),h_{2}(u),h_{2}(v)$ are distinct.  From \cite{Reilly1} it follows that there exists a homomorphism $\psi:\mathcal{X^{+}}\rightarrow B_{2}$ such that $\psi(u)\neq 0$ and $\phi(h(u))=(0,1)$ (since $(h(u),0),(h(u),1)$ does not lie in the same connected component of $Gr(u)$). Let $\psi(h_{2}(u))=(\alpha_{1},\beta_{1}),\psi(h_{2}(v))=(\alpha_{2},\beta_{2}), \psi(x_{i})=(l_{i},r_{i})$ for any letter $x_{i}\in C(u)\backslash\{h(u),h_{2}(u),h_{2}(v)\}$ and $\psi(u)=(0,y)$. Consider a homomorphism $\phi:\mathcal{X^{+}}\rightarrow W_{\lambda}$ defined as follows

$h(u)=h(v)\mapsto (0,1,2,0),$

$h_{2}(u)\mapsto (\alpha_{1},\beta_{1},0,1),$

$h_{2}(v)\mapsto (\alpha_{2},\beta_{2},1,1),$

$x_{i}\mapsto (l_{i},r_{i},1,1)$ for any $x_{i}\in C(u)\backslash\{h(u),h_{2}(u),h_{2}(v)\}.$

We obtain that $\phi(u)=(0,y,0,1)$ and $\phi(v)=(0,y,1,1)$ i.e. $\phi(u)\neq\phi(v)$. This contradiction proves the necessity.

\textbf{Sufficiency.} Let the identity $u\approx v$ satisfies Conditions \emph{(1)},\emph{(2)}. Let us show that the identity $u\approx v$ holds in the variety $\mathbf{W_{\lambda}}$. Consider any homomorphism $\phi:\mathcal{X^{+}}\rightarrow W_{\lambda}$. Since the identity $u\approx v$ holds in the variety $\mathbf{B_{2}}$ we find $\tau(\phi(u))=\tau(\phi(v))$. Thus $\phi(u),\phi(v)$ lies in the same equivalence class defined by $\tau$. One can assume that $\phi(u),\phi(v)\neq0$ since equivalence class of [0] consists of one element. Let $\phi(u)=(a_{1},a_{2},a_{3},a_{4})$ and $\phi(v)=(b_{1},b_{2},b_{3},b_{4})$. Since $\tau(\phi(u))=\tau(\phi(v))$ we find that $a_{1}=b_{1}$ and $a_{2}=b_{2}$. Further, $N_{2}\in \mathbf{B_{2}}\subseteq \mathbf{W_{\lambda}}$ implies that $a_{4}=b_{4}=1$. Let us show that $a_{3}=b_{3}$. Let $\phi(h(u))=(a_{1},p,y,z)$ and $\phi(h_{2}(u))=(p,x_{1},y_{1},z_{1})$. If $(h(u),0),(h(u),1)$ lies in the same connected component of $Gr(u)$ then $a_{3}=b_{3}=x$. If $(h(u),0),(h(u),1)$ does not lie in the same connected component of $Gr(u)$ then $h_{2}(u)=h_{2}(v)$. If $y\neq2$ then $a_{3}=b_{3}=y$. If $y=2$ then $y_{1}\neq 2$ (because otherwise we would obtain that $\phi(h(u)h_{2}(u))=0$). Hence $\phi(h(u)h_{2}(u))=(a_{1},x_{1},y_{1},1)$. So $a_{3}=b_{3}=y_{1}$. In both cases we find that $a_{3}=b_{3}$. Hence $\phi(u)=\phi(v)$ i.e. the identity $u\approx v$ holds in the variety $\mathbf{W_{\lambda}}$.
\end{proof}

Now we are ready to find a finite identity basis of the variety $\mathbf{W_{\lambda}}$.

\begin{lemma}
\label{WIBSFB}
   The identities \eqref{a2=a3}, \eqref{aba=ababa} and
\begin{align}
\label{abc2d2=abd2c2}
& abc^{2}d^{2}\approx abc^{2}d^{2}\\
\label{a2b2=ab2a}
& a^{2}b^{2}\approx ab^{2}a
\end{align}
form a identity basis of the variety $\mathbf{W_{\lambda}}$.
\end{lemma}

\begin{proof} Note that the identity
\begin{equation}
\label{a2ba=aba2}
a^{2}ba\approx aba^{2}
\end{equation}
follows from identities \eqref{aba=ababa}, \eqref{a2b2=ab2a}. Indeed,
$$a^{2}ba\approx^{\eqref{aba=ababa}} (a^{2})(baba)\approx^{\eqref{a2b2=ab2a}} (ababa)a\approx^{\eqref{aba=ababa}} aba^{2}.$$

Consider an identity $au\approx av$ which holds in the variety $\mathbf{W_{\lambda}}$. Note that Corollary \ref{B2IDC} implies that $a\in C(u)$ if and only if $a\in C(v)$. Consider possible cases:
\begin{arabiclist}
\item $a\notin C(u)$. Then $a\notin C(v)$ and $(h(u),0),(h(u),1)$ does not lie in the same connected component of the graph $Gr(u)$. Then by Lemma \ref{WIBSID} we obtain $h_{2}(u)=h_{2}(v)$ i.e. the identity $au\approx av$ has a form $abu'\approx abv'$. By Corollary \ref{B2IDC} we can observe that identity $bu'\approx bv'$ holds in the variety $\mathbf{B_{2}}$. Then the identity $bu'\approx bv'$ holds in the variety $\mathbf{L_{2}}\vee\mathbf{B_{2}}$ and follows from identities \eqref{a2=a3}, \eqref{aba=ababa}, $bc^{2}d^{2}\approx bd^{2}c^{2}$. In this case we find that the identity $abu'\approx abv'$ follows from identities \eqref{a2=a3}, \eqref{aba=ababa}, \eqref{abc2d2=abd2c2}.
    \item $a\in C(u)$. Then we have $a\in C(v)$ and the identity $au\approx av$ has a form $au'au''\approx av'av''$. Consider possible cases:
\begin{alphlist}
\item $(h(u),1),(h(u),1)$ lies in the same connected component of the graph $Gr(u)$. Applying the identity \eqref{aba=ababa} we obtain the identity $au'au'au''\approx av'av'av''$. Note that the identities $au'au''\approx au'a^{2}u''\approx av'av''\approx av'a^{2}v''$ holds in the variety $\mathbf{B_{2}}$ (by Lemma \ref{B2ID}) and follows from identities \eqref{a2=a3}, \eqref{aba=ababa}, $bc^{2}d^{2}\approx bd^{2}c^{2}$ (by Proposition \ref{B2L2FB}). Hence the identity $au'au'a^{2}u''\approx av'av'a^{2}v''$ follows from identities \eqref{a2=a3}, \eqref{aba=ababa}, \eqref{abc2d2=abd2c2}. Now applying the identity $a^{2}ba\approx aba^{2}$ we obtain the identity $a^{2}u'au'au''\approx a^{2}v'av'av''$. Applying the identity \eqref{aba=ababa} we obtain the identity $a^{2}u'au''\approx a^{2}v'av''$. As we noted the identity $au'au''\approx av'av''$ follows from identities \eqref{a2=a3}, \eqref{aba=ababa}, $bc^{2}d^{2}\approx bd^{2}c^{2}$. Hence the identity $au\approx av$ follows from identities \eqref{a2=a3}, \eqref{aba=ababa}, \eqref{abc2d2=abd2c2}, \eqref{a2b2=ab2a}.
      \item $(h(u),1),(h(u),1)$ does not lie in the same connected component of the graph $Gr(u)$. Then $h_{2}(u)=h_{2}(v)$ by Lemma \ref{WIBSID}. Thus the identity $au'au''\approx av'av''$ has a form $ab\overline{u}au''\approx ab\overline{v}au''$. Applying the identity \eqref{aba=ababa} we obtain the identity $ab\overline{u}ab\overline{u}au''\approx ab\overline{v}ab\overline{v}au''$. But the identity $b\overline{u}ab\overline{u}au''\approx b\overline{v}ab\overline{v}au''$ holds in the variety $\mathbf{B_{2}}$ by Lemma \ref{B2ID} and hence follows from identities \eqref{a2=a3}, \eqref{aba=ababa}, $bc^{2}d^{2}\approx bd^{2}c^{2}$. Hence the identity $ab\overline{u}ab\overline{u}au''\approx ab\overline{v}ab\overline{v}au''$ follows from identities \eqref{a2=a3}, \eqref{aba=ababa}, \eqref{abc2d2=abd2c2}.
\end{alphlist}
\end{arabiclist}
We obtained that any identity $u\approx v$ which holds in the variety $\mathbf{W_{\lambda}}$ follows from identities \eqref{a2=a3}, \eqref{aba=ababa}, \eqref{abc2d2=abd2c2}, \eqref{a2b2=ab2a}. Thus they are constitute a identity basis of the variety $\mathbf{W_{\lambda}}$.
\end{proof}

\section*{Acknowledgments}

The author would like to express his gratitude to M. V. Volkov for his
time and patience in responding to the author's questions, to S. I. Kublanovsky for suggesting to investigate indicator Burnside semigroups and to N. A. Vavilov and B. M. Vernikov for their valuable remarks and suggestions.

\appendix

\bibliographystyle{amsplain}

\end{document}